\documentclass[11pt]{amsart}
\usepackage{amsfonts}
\usepackage{amsmath}
\usepackage{amsthm}
\usepackage{amssymb}
\usepackage{graphicx}
\usepackage{enumerate}
\usepackage{multicol}
\usepackage{mathrsfs}
\usepackage[usenames,dvipsnames]{pstricks}
\usepackage{epsfig}
\usepackage{pst-grad} 
\usepackage{pst-plot} 

\usepackage[utf8]{inputenc}

\newtheorem{prop}{Proposition}
\newtheorem{theorem}{Theorem}

\DeclareMathOperator{\Id}{Id}
\DeclareMathOperator{\n}{\mathfrak{n}}
\DeclareMathOperator{\vv}{\mathfrak{v}}
\DeclareMathOperator{\z}{\mathfrak{z}}

\DeclareMathOperator{\End}{End}

\DeclareMathOperator{\ra}{\rangle}
\DeclareMathOperator{\la}{\langle}

\title{The Sub-Riemannian cut locus of $H$-type groups}
\author[Christian Autenried, Mauricio Godoy M.]{Christian Autenried\\Mauricio Godoy Molina}

\address{Department of Mathematics, University of Bergen, Norway.}
\email{christian.autenried@math.uib.no}

\address{Department of Mathematics, University of Bergen, Norway.}
\email{mauricio.godoy@math.uib.no}

\thanks{The authors are partially supported by the grants of the Norwegian Research Council \#204726/V30, \#213440/BG and NFR-FRINAT.}



\begin{document}

\maketitle

\begin{abstract}
In the present paper we give a proof of the fact that the sub-Riemannian cut locus of a wide class of nilpotent groups of step two, called $H$-type groups, starting from the origin corresponds to the center of the group. We obtain this result by completely describing the sub-Riemannian geodesics in the group, and using these to obtain three disjoint sets of points in the group determined by the number of geodesics joining them to the origin.
\end{abstract}

\section{Introduction}

The $H$(eisenberg)-type Lie algebras, which are one of the most important examples of nilpotent Lie algebras of step 2, were introduced by A. Kaplan in his foundational work~\cite{Ka1}. Their Lie algebra structure is intimately related to the existence of certain Clifford algebra representations, which we will introduce carefully later on. These Lie algebras have a deep connection to sub-Riemannian geometry, which we will proceed to explain.

Let us first recall generalities about sub-Riemannian manifolds. A triplet $(Q,{\mathcal H},\langle\cdot \,,\cdot\rangle)$, where ${\mathcal H}\hookrightarrow TM$ is a distribution, i.e., a subbundle of the tangent bundle of $M$, and $\langle\cdot \,,\cdot\rangle$ is a fiber inner product defined on ${\mathcal H}$, is called a sub-Riemannian manifold. For most applications, it is assumed that the distribution ${\mathcal H}$ is bracket-generating, that is,
\[
{\rm Lie}\,{\mathcal H}=\mbox{ Lie algebra generated by sections of }{\mathcal H}=\Gamma(TM).
\]
The step of ${\mathcal H}$ is, by convention, the minimal length of brackets needed to generate all the vector fields on $M$ plus one. An absolutely continuous curve $\gamma\colon[0,1]\to Q$ is called horizontal if $\dot\gamma(t)\in{\mathcal H}_{\gamma(t)}$ almost everywhere. For horizontal curves, we can define their length as usual
\begin{equation}\label{eq:length}
L(\gamma)=\int_0^1 \sqrt{\langle\dot\gamma(t),\dot\gamma(t)\rangle} dt.
\end{equation}
A horizontal curve $\tilde\gamma\colon[0,1]\to Q$ is a sub-Riemannian geodesic if it locally minimizes the functional~\eqref{eq:length}, i.e., there exists a partition $0=t_0<t_1<\cdots<t_n=1$ of $[0,1]$ such that for $i=0,\dotsc,n-1$
\[
L(\tilde\gamma|_{[t_i,t_{i+1}]})=\inf\{L(\gamma)\colon\gamma\mbox{ is horizontal,}\gamma(0)=\tilde\gamma(t_i),\gamma(1)=\tilde\gamma(t_{i+1})\}.
\]

It is well-known that for a nilpotent algebra ${\mathfrak n}={\mathfrak v}\oplus{\mathfrak z}$ there is a unique (up to isomorphism) connected and simply connected Lie group $N$ with Lie algebra ${\mathfrak n}$. Applying this idea to an $H$-type algebra, and by left-translating the subspace ${\mathfrak v}$ of ${\mathfrak n}$, we obtain a bracket-generating distribution ${\mathcal H}\hookrightarrow TN$ of step $2$. Any inner product defined on ${\mathfrak v}$ induces a sub-Riemannian metric on ${\mathcal H}$. Explicit equations for the sub-Riemannian geodesics in $H$-type groups can be found in~\cite{GKM}.

A fundamental tool in the analysis of sub-Riemannian manifolds is the so-called cut locus. Recall that the (sub-Riemannian) cut locus of $(Q,{\mathcal H},\langle\cdot \,,\cdot\rangle)$, with respect to a point $p\in Q$ is defined as the set of points $q\in Q$ such that there is more than one minimizing sub-Riemannian geodesic connecting $p$ to $q$. One of the most relevant applications where knowledge of this set plays an important part, is to describe the short-time asymptotic behavior of the heat kernel associated to a naturally defined sub-elliptic operator, see~\cite{BBN}. The aim of this paper is to give a complete characterization of the sub-Riemannian cut locus for the $H$-type groups.

This paper is organized as follows. In Section~\ref{sec:generalities}, we briefly recall some necessary definitions and give a precise form of the sub-Riemannian geodesics on an $H$-type group. In Section~\ref{section:sRcutlocus}, we prove the main result of this paper, which follows from studying carefully three different sets of points in the group. In order to understand these sets completely, we need to obtain results similar to those in~\cite{CCM}, but valid in the full generality of $H$-type groups.

\section{Sub-Riemannian geodesics on $H$-type groups}\label{sec:generalities}

\subsection{Sub-Riemannian $H$-type Lie groups}

Let us first recall the construction of the $H$-type Lie algebras $\n_{r}$. In what follows, all inner products are positive definite and $Cl_r$ is the Clifford algebra generated by the vector space $\z_{r}={\mathbb R}^r$ with inner product $\la \cdot \,, \cdot \ra_{\z_{r}}$.

Consider a $Cl_r$-module $\vv_{r}$, where $J\colon \z_{r} \to\End(\vv_{r})$ is the corresponding Clifford algebra representation. If $\la \cdot \,, \cdot \ra_{\vv_{r}}$ is an inner product on $\vv_{r}$, we can define a Lie bracket by
\begin{equation*}
 \langle J_Zv , w \rangle_{\vv_{r}} = \langle Z , [v, w]\rangle_{\z_{r}},\quad \text{for all }v,w \in \vv_{r}.
\end{equation*}
We set all brackets with elements in ${\z_{r}}$ to be zero. This induces a Lie algebra structure of step 2  on $\n_{r}={\vv_{r}}\oplus\z_{r}$. We define the inner product $\la \cdot \,, \cdot \ra:= \la \cdot \,, \cdot \ra_{\vv_{r}}+\la \cdot \,, \cdot \ra_{\z_{r}}$ on $\n_{r}$. Then 
\[
J_{Z'}J_Z+J_ZJ_{Z'}=-2\langle Z , Z' \rangle_{\z_r} \Id_{\vv_{r}} \mbox{ for all }Z \in \z_{r}.
\]
As it is usual in the literature, we call $\vv_r$ the horizontal space and the center $\z_r$ the vertical space.

Let  $\{v_1, \dotso ,v_m\}$ and $\{Z_1,\dotso,Z_n\}$ be orthonormal bases of $\vv_r$ and $\z_r$ respectively. The structure constants $C^k_{i j}$ and the coefficients $B^k_{i j}$ of the representation $J$ are defined by
\begin{equation*}
[v_{i} \,, v_{j} ] = \sum_{k=1}^n{C^k_{i j}Z_k} \qquad \text{and} \qquad J_{Z_k}v_{i}=\sum_{j=1}^m{B^k_{i j} v_{j}}.
\end{equation*}
It is easy to see that $B^{k}_{i j} =C^k_{i j}$, and we will use this fact freely throughout many of the forthcoming computations.

It follows that the structure matrices  $\{C^1, \dotso,C^n\} \subset \mathfrak{so}(m)$ defined by $C^k=(C_{ij}^k)_{ij}$ for all $k=1, \dotso,n$ satisfy the relations
$$C^kC^p=-C^pC^k, \qquad \text{ for } k \not=p,$$
and $(C^k)^2=-\Id_{\vv_r}$. Furthermore, we note that $C^kC^p \in \mathfrak{so}(m)$ for $k\not=p$, since
$$(C^kC^p)^T=(C^p)^T(C^k)^T=(-C^p)(-C^k)=C^pC^k=-C^kC^p,$$
where $(C^k)^T$ is the transposed matrix of $C^k$.

The $H$-type group $N_r$ is the unique (up to homomorphism) connected and simply connected Lie group with Lie algebra $\n_r$. The subspace $\vv_r$ defines a bracket generating distribution of step 2 over $N_r$ by left-translation, and the translations of the inner product $\la \cdot \,, \cdot \ra_{\vv_{r}}$ makes $N_r$ into a sub-Riemannian manifold. 

Recall that for a nilpotent Lie group, the exponential map is a diffeomorphism, see for example~\cite{CG}. Therefore we can identify $N_r$ with $\vv_r\oplus\z_r$. We will use this identification through this paper. Under this identification, we denote by $V_{r}$ and $Z_{r}$  the image of $\vv_r$ and $\z_r$ respectively.

\subsection{Sub-Riemannian geodesics on $N_r$}\label{ssec:geod}

From now on, we write a horizontal sub-Riemannian geodesic $c \colon [0,1] \to N_{r}$ by $c(t)=(x(t),z(t))$, where $x(t)$ is in $V_{r}$ and $z(t)=(z^1(t),\dotsc,z^n(t))$ is in $Z_{r}$. 

From~\cite[Theorem 2]{GKM}, we have the following formulas:
\begin{align}
x(t)&=\frac{\sin(t\vert \theta \vert)}{\vert \theta \vert}\,\dot x(0)+\frac{(1-\cos(t\vert \theta \vert))}{\vert \theta \vert^2}\Omega \dot x(0),\label{subgeodgeneralx}
\end{align}
\begin{align}
\dot z^{k}(t)&=\frac{1}{2}\dot x(0)^T \Bigg[\frac{C^{k}}{\vert \theta \vert} \cos(t\vert \theta \vert) \sin(t\vert \theta \vert)+\frac{C^k \Omega}{\vert \theta \vert^2} \cos(t\vert \theta \vert) (1-\cos(t\vert \theta \vert))\\ 
\nonumber &+\frac{\Omega^TC^k}{\vert \theta \vert^2}\sin^2(t\vert \theta \vert)+\frac{\Omega^T C^k\Omega}{\vert \theta \vert^3}\sin(t\vert \theta \vert) (1-\cos(t\vert \theta \vert)) \Bigg]\dot x(0),\label{subgeodgeneralt}
\end{align}
where $\theta=(\theta_1, \dotso, \theta_n)\not=(0, \dotso,0)$ is a vector of parameters coming from the Hamiltonian formulation, $\Omega=\sum_{k=1}^n{C^k \theta_k}$, $\theta=(\theta_1, \dotso, \theta_n)\not=(0, \dotso,0)$ and $\vert \theta \vert={\left (\sum_{k=1}^n{\theta_k^2} \right )}^{1/2}$.

For $\theta=(0, \dotso,0)$ the sub-Riemannian geodesics are straight lines, i.e., the geodesic starting at the point $(0,0)$ and reaching the point $(x,0)$ at time $t=1$ is given by $c(t)=(tx,0)$.

\begin{prop}\label{prop:vert}
The vertical part $z(t)$ of a horizontal geodesic $c(t)$ for the $H$-type group $N_{r}$, corresponding to $\theta=(\theta_1, \dotso, \theta_n)\not=(0, \dotso,0)$, is given by
\begin{equation}\label{eq:vert}
z(t)=\frac{  \vert \dot x(0) \vert^2}{2\vert\theta \vert^2}\left(t-\frac{\sin(t \vert\theta \vert)}{\vert\theta \vert} \right)\theta.
\end{equation}
\end{prop}

\begin{proof}
First we note that given a skew-symmetric matrix $A \in \mathfrak{so}(m)$, i.e. $A^T=-A$, then $v^TAv=0$ for any vector $v \in \mathbb{R}^m$ as
$$\mathbb{R} \ni v^TAv=(v^TAv)^T=v^TA^Tv=-v^TAv.$$
Now we calculate the matrices $C^k \Omega$, $\Omega^TC^k$ and $\Omega^T C^k\Omega$.
\begin{align*}
C^k \Omega&= C^k\sum_{l=1}^n{C^l \theta_l}=\sum_{l=1}^n{C^kC^l \theta_l}=- \theta_k\Id_{\vv_r}+\sum_{l\not=k}^n{C^kC^l \theta_l}, \\
\Omega^TC^k&=-\Omega C^k=-\sum_{l=1}^n{C^l C^k \theta_l}=\theta_k \Id_{\vv_r}-\sum_{l\not=k}^n{C^lC^k \theta_l} \\
&=\theta_k \Id_{\vv_r}+\sum_{l\not=k}^n{C^kC^l \theta_l}
\end{align*}
\begin{align*}
\Omega^T C^k\Omega&= \left(\theta_k \Id_{\vv_r}+\sum_{l\not=k}^n{C^kC^l \theta_l}\right) \Omega=\theta_k \Omega+\sum_{l\not=k}^n{\sum_{p=1}^n{C^kC^lC^p \theta_l \theta_p}} \\
&= \theta_k \Omega+\sum_{l\not=k}^n{\sum_{p\not=k}^n{C^kC^lC^p \theta_l \theta_p}}+\sum_{l\not=k}^n{C^kC^lC^k \theta_l \theta_p} \\
&=\theta_k \Omega+\sum_{l\not=k}^n{\sum_{p\not=k}^n{C^kC^lC^p \theta_l \theta_p}}+\sum_{l\not=k}^n{C^l \theta_l \theta_p} \\
&=\theta_k \Omega+\sum_{l\not=k}^n{-C^k \theta_l^2}+\sum_{l\not=k}^n{C^l \theta_l \theta_p},
\end{align*}
where the last equation is obtained as $C^lC^p=-C^pC^l$ for $p\not=l$. As $\Omega$, $C^p$, $C^pC^l$ are skew-symmetric for any $p\not=l$, we obtain that
\begin{align*}
\dot x(0)^TC^k \dot x(0)&=0, \\
\dot x(0)^TC^k \Omega \dot x(0)&=- \theta_k |\dot x(0)|^2+\sum_{l\not=k}^n{\dot x(0)^TC^kC^l\dot x(0) \theta_l}=- \theta_k |\dot x(0)|^2, \\
\dot x(0)^T\Omega^TC^k \dot x(0)&=\theta_k  |\dot x(0)|^2+\sum_{l\not=k}^n{\dot x(0)^TC^kC^l \dot x(0) \theta_l} =\theta_k  |\dot x(0)|^2\\
\end{align*}
\begin{align*}
\dot x(0)^T\Omega^T C^k\Omega \dot x(0)&=\theta_k \dot x(0)^T\Omega \dot x(0)+\sum_{l\not=k}^n{-\dot x(0)^TC^k \dot x(0) \theta_l^2}\\
&+\sum_{l\not=k}^n{\dot x(0)^TC^l \dot x(0) \theta_l \theta_p}=0.
\end{align*}
It follows that for any $k=1, \dotso,n$
\begin{align*}
\dot z^k(t)&= \theta_k |\dot x(0)|^2 \frac{-\cos(t\vert \theta \vert) (1-\cos(t\vert \theta \vert))+\sin^2(t\vert \theta \vert)}{{2\vert \theta \vert^2}} \\
&=\theta_k |\dot x(0)|^2\frac{1-\cos(t\vert \theta \vert)}{{2\vert \theta \vert^2}}, \\
z^k(t)&=\frac{\theta_k |\dot x(0)|^2}{2\vert\theta \vert^2}\left(t-\frac{\sin(t \vert\theta \vert)}{\vert\theta \vert} \right).\qedhere
\end{align*}
\end{proof}

This simplification will allow us to determine concretely the points in any $H$-type group where minimizing sub-Riemannian geodesics starting from the origin stop being unique.


\section{Sub-Riemannian cut locus of $H$-type groups}\label{section:sRcutlocus}


Let $(Q,\mathcal H,g_{\mathcal H})$ be a sub-Riemannian manifold. The (sub-Riemannian) cut locus of a point $q_0 \in Q$ is the set 
\begin{eqnarray*}
K_{q_0}=& \Big\{q \in Q \vert\  \text{there exist}\   T>0, \text{ and minimizing horizontal }
\\
&\text{ geodesics }\gamma_1,\gamma_2\colon [0,T] \to Q,\ \gamma _{1}\neq \gamma _{2},  \text{ such that}\
\\
&   \gamma _{1}(0)= \gamma _{2}(0)=q_0\  \text{and}\  \gamma _{1}(T)= \gamma _{2}(T) =q\Big\} .
\end{eqnarray*}
In this section, we give a precise description of the sub-Riemannian cut locus of curves starting from the identity $(0,0)$ in the $H$-type groups introduced previously. More precisely, we want to prove the following

\begin{theorem}\label{th:main}
The cut locus $K_{(0,0)}$ of the $H$-type group $N_r$ is given by the points of the form $(0,z)$.
\end{theorem}  

This result will be achieved in three steps: points of the form $(0,z)$ are in $K_{(0,0)}$; points of the form $(x,z)$, with $x\neq0$ and $z\neq0$, are not in $K_{(0,0)}$; and neither are the points of the form $(x,0)$.

\subsection{The vertical space is contained in the cut locus}

As a first step, we show that the points $(0,z)\in N_r$ are in the cut locus. The geodesics connecting the origin $(0,0)$ and $(0,z)$ and their length are given by the following Theorem which generalizes~\cite[Theorem 6.3]{CCM} to arbitrary $H$-type groups.

\begin{theorem}
For each natural number $k\in{\mathbb N}$, there exists a sub-Riemannian geodesic $c_k(t)=(x_k(t),z_k(t))$ in $N_r$ joining the origin with the point $(0,z)$. These curves have lengths $l_1,l_2, \dotso$, where $l^2_k=4 k \pi  \vert z \vert$, $k \in \mathbb{N}$, and their equations are
\begin{eqnarray*}
x_k(t)=4 \frac{\sin^2(k \pi t)}{ \vert \dot x(0) \vert^2 } \mathcal{Z}\dot{x}(0)  + \frac{\sin (2 k \pi  t)}{2 k \pi} \dot{x}(0), \quad k \in \mathbb{N},
\end{eqnarray*}
where $\mathcal Z = \sum_{r=1}^n{z^rC^r }$ and 
\begin{eqnarray*}
z_k(t)= \left (t-\frac{\sin(2 \pi k t)}{2 \pi k} \right ) z, \quad k \in \mathbb{N}.
\end{eqnarray*}
\end{theorem}

\begin{proof}
We follow a similar scheme as in~\cite{CCM}. Evaluating equation~\eqref{subgeodgeneralx} at $t=1$, and after some simple computations, we see that
\[
0=|x(1)|^2=\frac{4|\dot x(0)|^2}{|\theta|^2}\sin^2\left(\frac{|\theta|}2\right).
\]
Since we can assume that $|\dot x(0)|\neq0$, it follows that $|\theta|=2k\pi$, for $k\in{\mathbb N}$. This, in turn, implies that
\[
z=z(1)=\frac{|\dot x(0)|^2}{8k^2\pi^2}\theta,
\]
thus $|\dot x(0)|^2=4k\pi|z|$. We immediately obtain the geodesic equations.

The length of the geodesics follows easily, since
\[
l_k^2=l(c_k)^2=\left(\int_0^1\sqrt{|\dot x(t)|}dt\right)^2=|\dot x(0)|^2=4k\pi|z|.\qedhere
\]
\end{proof}

Suppose $c(t)$ is the minimizing geodesic between the origin and $(0,z)$ with length $4 \pi \vert z \vert$ and with initial vector $\dot{x}(0)\not=0$. We define the geodesic $\tilde c(t)=(\tilde x(t),\tilde z(t))$ with initial vector $-\dot{x}(0)$ by
\begin{eqnarray*}
\tilde x(t)&=& -4\frac{\sin^2(\pi m t)}{ \vert \dot x(0) \vert^2 } \mathcal{Z}\dot{x}(0)  -\frac{\sin (2 \pi t)}{2 \pi }  \dot{x}(0),  \\
\tilde z(t)&=& \left (t-\frac{\sin(2 \pi  t)}{2 \pi }\right )z.
\end{eqnarray*}
This geodesic is minimizing between the origin and $(0,z)$ as it has length $4 \pi \vert z \vert$ and clearly $c \not = \tilde c$. It follows that $(0,z)$ is an element of the cut locus of the origin.

\subsection{If $x \not =0$ and $z\not=0$, then $(x,z)$ is not in the cut locus}

Theorem~\ref{th:main} will follow after proving that points not contained in the vertical space are not elements in the cut locus. We start the analysis by proving the following extension of~\cite[Theorem 6.5]{CCM}.

\begin{theorem}\label{geo}
Given a point $(x,z)\in N_r$ with $x\not=0$, $z\not=0$, there are finitely many sub-Riemannian geodesics joining the origin $(0,0)$ with $(x,z)$. Let $\vert \theta \vert_1, \dotso , \vert \theta \vert_N$ be solutions of the equation
\begin{eqnarray*}
\frac{4 \vert z \vert}{\vert x \vert^2}=\mu(\vert \theta \vert/2 ),
\end{eqnarray*}
where $\mu(\alpha)= \frac{\alpha}{\sin^2(\alpha)}-\cot(\alpha)$. Then the equation of the geodesic $c_k(t)=(x_k(t),z_k(t))$, $t \in [0,1]$, corresponding to $|\theta|_k$, is
\begin{align*}
x_k(t) &= \sin\left(\frac{t\vert \theta \vert_k}2\right) \cos\left(\frac{t\vert \theta \vert_k}2\right) \left ( \frac{8 \sin^2\left(\tfrac{\vert \theta \vert_k}2\right)\left(\tan\left(\frac{t\vert \theta \vert_k}2\right)\cot\left(\frac{\vert \theta \vert_k}2\right)-1\right)}{|x|^2(\vert \theta \vert_k-\sin(\vert \theta \vert_k))} {\mathcal Z} x\right.  \\
&+ \left.\left(\tan\left(\frac{t\vert \theta \vert_k}2\right)+ \cot\left(\frac{\vert \theta \vert_k}2\right)\right)x \right) , \\\
z_k(t)&=\frac{t|\theta|_k-\sin(t|\theta|_k)}{|\theta|_k-\sin(|\theta|_k)}\;z
\end{align*}
with $\mathcal Z = \sum_{r=1}^n{z^rC^r }$ and $k=1,2,\dotso ,N$. The lengths of these geodesics are $l_k^2= \nu (\vert \theta \vert_k)(\vert x \vert ^2 + 4 \vert z \vert)$, where 
$$ \nu(\alpha) = \frac{\alpha^2}{2(1+\alpha-\cos(\alpha)-\sin(\alpha))}.$$
\end{theorem}

\begin{proof}
Putting $s=1$ into equations~\eqref{subgeodgeneralx} and~\eqref{eq:vert}, we see that
\begin{align}
|x|^2&=|x(1)|^2=\frac{4\sin^2(|\theta|/2)}{|\theta|^2}\,|\dot x(0)|^2,\label{eq:xnot0}\\
z&=z(1)=\frac{|x|^2}{8\sin^2(|\theta|/2)}\left(1-\frac{\sin(|\theta|)}{|\theta|}\right)\theta=\frac{|x|^2\mu(|\theta|/2)}{4|\theta|}\,\theta.\label{eq:znot0}
\end{align}
It follows that $|z|=\frac14|x|^2\mu(|\theta|/2)$. Let $\vert \theta \vert_1, \dotso , \vert \theta \vert_N$ be the solutions of this equation. We fix a solution $|\theta|_k$ and obtain by the use of equation~\eqref{eq:xnot0} in~\eqref{eq:znot0} that
\begin{equation}\label{eq:thetaandz}
\theta=\frac{2|\theta|_k^3\:z}{|\dot x(0)|^2(|\theta|_k-\sin(|\theta|_k))},
\end{equation}
and therefore we have that
\[
z_k(t)=\frac{t|\theta|_k-\sin(t|\theta|_k)}{|\theta|_k-\sin(|\theta|_k)}\;z.
\]

To find the expression for $x_k(t)$, let us first observe that
\[
\left(\frac{|\theta|_k}2\cot\left(\frac{|\theta|_k}2\right)\Id_{\vv_r}-\frac{\Omega}2\right)
\left(\frac{\sin(|\theta|_k)}{|\theta|_k}\Id_{\vv_r}+\frac{1-\cos(|\theta|_k)}{|\theta|_k^2}\,\Omega\right)=\Id_{\vv_r},
\]
which follows from simple trigonometric identities and the fact that choosing the corresponding covector $\theta$, we have that $\Omega^2=-|\theta|_k^2\Id_{\vv_r}$. Since equation~\eqref{subgeodgeneralx} can be written as
\[
x=\left(\frac{\sin(\vert \theta \vert_k)}{\vert \theta \vert_k}\,\Id_{\vv_r}+\frac{(1-\cos(\vert \theta \vert_k))}{\vert \theta \vert_k^2}\Omega\right) \dot x(0),
\]
then $\dot x(0)=\left(\frac{|\theta|_k}2\cot\left(\frac{|\theta|_k}2\right)\Id_{\vv_r}-\frac{\Omega}2\right)x$, and therefore
\begin{align*}
x_k(t)&=\left(\frac{\sin(t\vert \theta \vert_k)}{\vert \theta \vert_k}\,\Id_{\vv_r}+\frac{(1-\cos(t\vert \theta \vert_k))}{\vert \theta \vert_k^2}\Omega \right)\dot x(0)\\
&=\left(\frac{\sin(t\vert \theta \vert_k)}{\vert \theta \vert_k}\,\Id_{\vv_r}+\frac{(1-\cos(t\vert \theta \vert_k))}{\vert \theta \vert_k^2}\Omega \right)\left(\frac{|\theta|_k}2\cot\left(\frac{|\theta|_k}2\right)\Id_{\vv_r}-\frac{\Omega}2\right)x.
\end{align*}
The equation for $x_k(t)$ in the statement follows from a simple computation, using the formula above and equation~\eqref{eq:thetaandz}.

We calculate the length of our obtained geodesics. For a fixed solution $|\theta|_k$ we obtain 
\begin{align*}
\dot x_k(t)&= \frac{1}{2} \Big ( \Big ( -\cos(t |\theta|_k)+\cot \left ( \frac{|\theta|_k}{2} \right ) \sin(t |\theta|_k) \Big) \Omega x \\
&+ |\theta|_k \Big (\cos(t |\theta|_k) \cot \left ( \frac{|\theta|_k}{2} \right ) + \sin(t |\theta|_k) \Big )x \Big ),
\end{align*}
and therefore
\begin{align*}
\langle \dot x_k(t) \,, \dot x_k(t) \rangle&= \langle \Omega x \,, \Omega x \rangle \frac{1}{4} \Big ( -\cos(t |\theta|_k)+\cot \left ( \frac{|\theta|_k}{2} \right ) \sin(t |\theta|_k) \Big)^2 \\
&+ \langle x \,, x \rangle \frac{1}{4} |\theta|^2_k \Big (\cos(t |\theta|_k) \cot \left ( \frac{|\theta|_k}{2} \right ) + \sin(t |\theta|_k) \Big )^2 \\
&= \langle x \,, x \rangle \frac{1}{4} |\theta|^2_k \Big ( \Big(  -\cos(t |\theta|_k)+\cot \left ( \frac{|\theta|_k}{2} \right ) \sin(t |\theta|_k) \Big)^2 \\
&+ \Big (\cos(t |\theta|_k) \cot \left ( \frac{|\theta|_k}{2} \right ) + \sin(t |\theta|_k) \Big )^2 \Big) \\
&= \langle x \,, x \rangle \frac{|\theta|_k^2}{2-2\cos(|\theta|_k)}\,.
\end{align*} 
Hence $l_k^2= \nu (\vert \theta \vert_k)(\vert x \vert ^2 + 4 \vert z \vert)$, since $4 \vert z \vert=|x|^2\mu(|\theta|/2)$.
\end{proof}

Given a point $(x,z)$ with $x \not =0$, $z\not=0$. Then there exists $N$ solutions $\vert \theta \vert_1, \dotso , \vert \theta \vert_N$ of the equation \begin{eqnarray*}
\frac{4 \vert z \vert}{\vert x \vert^2}=\mu \left (\frac{\vert \theta \vert}{2} \right).
\end{eqnarray*} 
If $N=1$, then there does not exist a second minimizing geodesic. Hence $(x,z)$ is not in the cut locus. 
\\
For $N>1$, we have to examine the solutions $\vert \theta \vert_k$ in detail. Without loss of generality, we assume that $\vert \theta \vert_k < \vert \theta \vert_{k+1}$. We know that $\mu$ is an increasing diffeomorphism on the interval $(-2\pi \,, 2\pi)$ onto $\mathbb{R}$ such that $\vert \theta \vert_1<2\pi$ and $\vert \theta \vert_k>2\pi$ for all $1<k\leq N$, see Figure~\ref{mu}. In Figure~\ref{vi}, we see that $\nu(2\pi)=\pi$ and that $\nu(x)<\nu(y)$ for all $x \in [0,\pi)$, $y\in(\pi, \infty)$. Hence 
$$\nu(\vert \theta \vert_1)< \nu(\vert \theta \vert_k), \qquad \text{ for } 1<k\leq N.$$
This implies that the geodesics $c_k(t)=(x_k(t),z_k(t))$ cannot be minimizing for $1<k\leq N$. This implies that the only minimizing geodesic between the origin and $(x,z)$ with $x\not=0$, $z \not=0$ is given by $c_1(t)=(x_1(t),z_1(t))$, hence $(x,z)$ is not in the cut locus.
 \begin{figure}[h]\label{mu}
 \centering
 \includegraphics[scale=0.75]{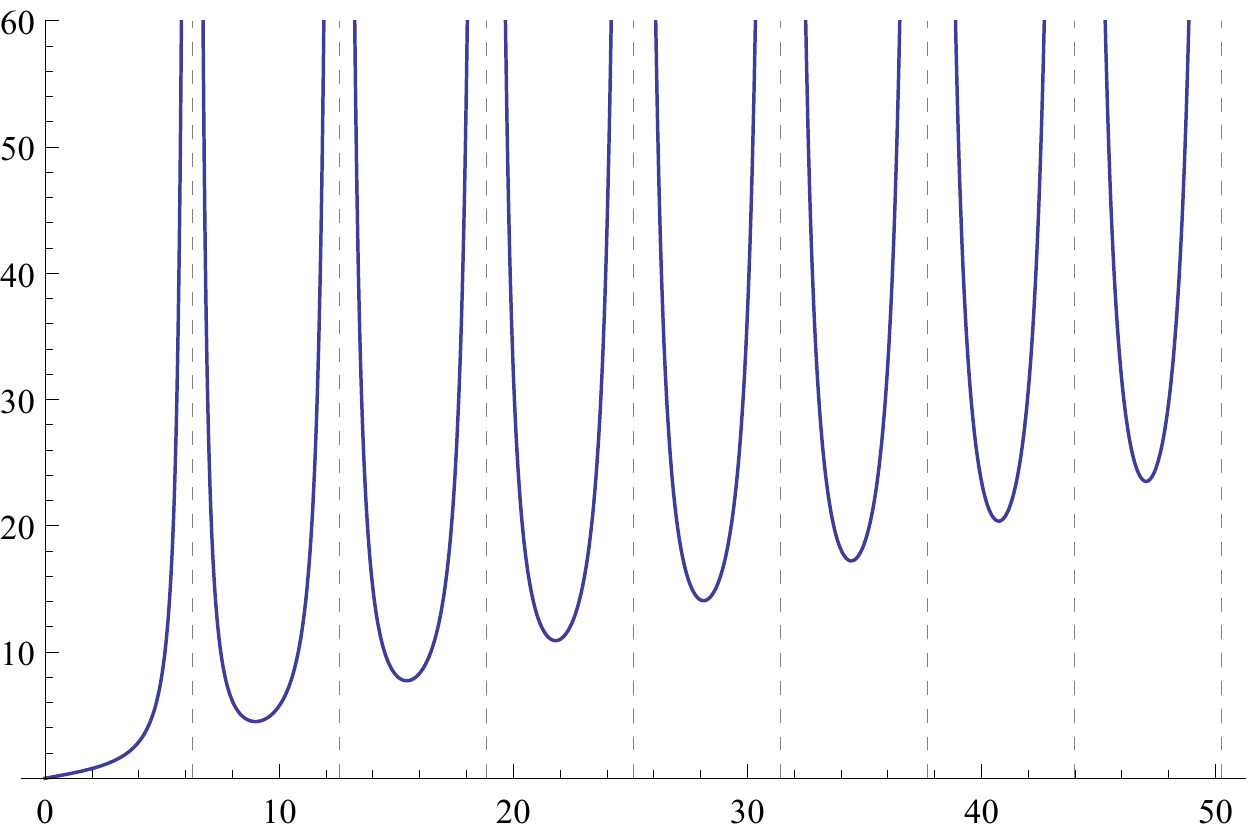}
 \caption{$\mu\left(\frac{\alpha}{2}\right)=\frac{\left ( \frac{\alpha}{2} \right)^2 }{\sin^2\left (\frac{\alpha}{2} \right)}-\cot\left(\frac{\alpha}{2}\right)$ on the interval $[0 \,, 16\pi ]$ with vertical lines at the points $2n\pi$, $n \in \mathbb{N}$.}
 \label{mu}
\end{figure}

\begin{figure}[h]\label{vi}
 \centering
 \includegraphics[scale=0.5]{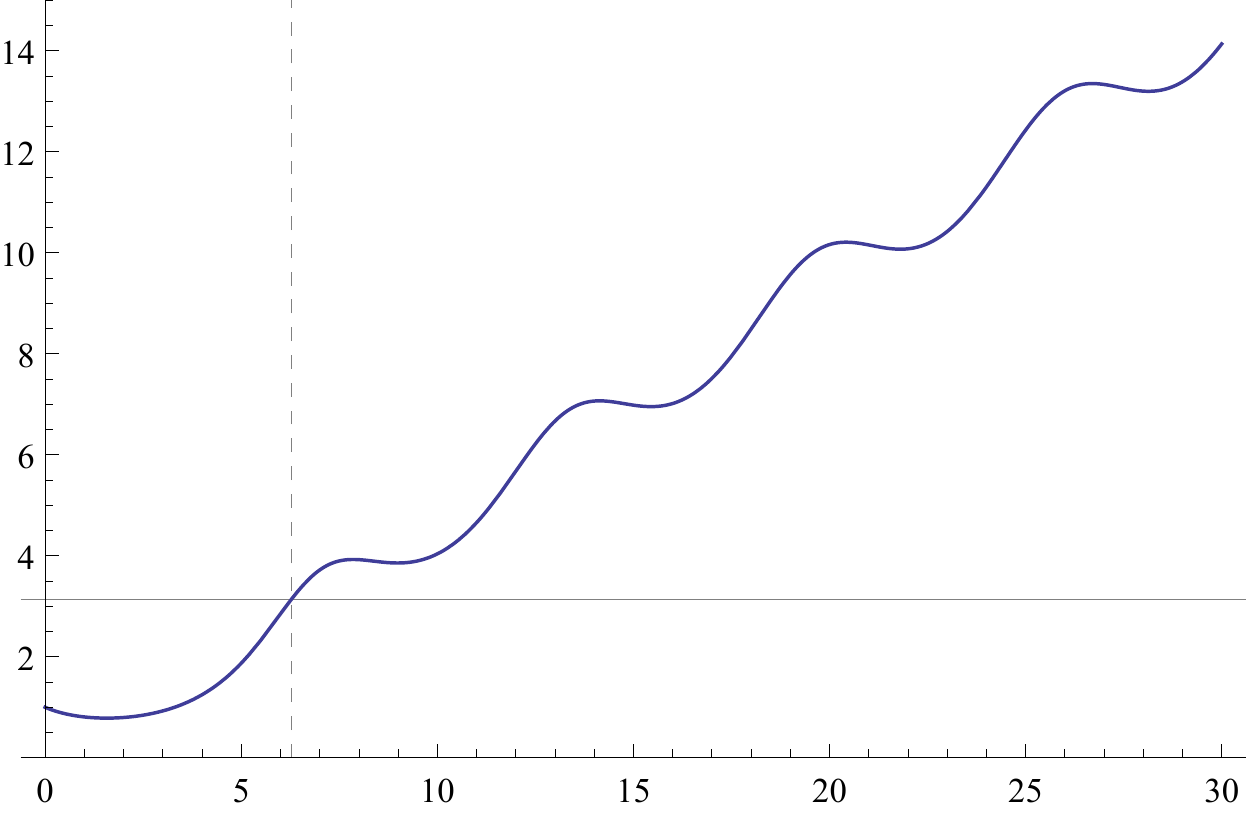}
 \caption{$\frac{\alpha^2 }{2(1+\alpha-\cos(\alpha)-\sin(\alpha))}$ on the interval $[0 \,, 30 ]$ with vertical line at the point $2\pi$ and horizontal line at the point $\pi$.}
 \label{vi}
\end{figure} 

\subsection{Points of the form $(x,0)$ are not in the cut locus}

To conclude the proof of Theorem~\ref{th:main}, we prove the following result.

\begin{theorem}\label{octx0}
A sub-Riemannian geodesic $c(t)$ in $N_r$ is horizontal with constant $z$-coordinate $z_0\in Z_r$ if and only if $c(t)=(at,z_0)$ for some vector $a \in V_r$ such that $|a|\not=0$.
\end{theorem}

\begin{proof}
Since $z(t)=z_0$ is constant, then $0=\dot z(t)$. If we assume $|\theta|\neq 0$, then we can apply Proposition~\ref{prop:vert}, to see that
\[
0=\dot z(t)=\frac{|\dot x(0)|^2(1-\cos(t\vert \theta \vert))}{{2\vert \theta \vert^2}}\theta.
\]
Since $|\dot x(0)|^2\neq 0$ and $(1-\cos(t\vert \theta \vert))\neq0$ for all $t\in[0,1]$, we obtain a contradiction. It follows that $\theta$ must vanish, and thus 
\[
x(t)=t\,\dot x(0),
\]
from the characterization of geodesics in Subsection~\ref{ssec:geod}. Setting $a=\dot x(0)$, the result is proved.
\end{proof}

It remains to show that there is no geodesic connecting the origin $(0,0)$ with $(x,0)$ with non-constant vertical component $z(t)$. Let assume that there exist such a geodesic which reaches $(x,0)$ at time $t_0=1$, then the non-constant part $z(t)$ is given by
\begin{eqnarray*}
z(t)=\frac{\vert \dot x (0) \vert^2}{4 \vert \theta \vert ^2}\left ( t - \frac{\sin(2t\vert \theta \vert )}{2\vert \theta \vert} \right ) \theta.
\end{eqnarray*}
It follows that $1=\frac{\sin(2\vert \theta \vert )}{2\vert \theta \vert}$ if and only if $\vert \theta \vert=0$, see Figure~\ref{sin}, which implies that $z(t)$ is constant. This is a contradiction to our assumption, hence there does not exist a geodesic connecting the origin $(0,0)$ with $(x,0)$ with non-constant vertical component $z(t)$. Hence the geodesic given in Theorem~\ref{octx0} is the unique geodesic connecting $(0,0)$ with $(x,0)$.

\begin{figure}[h]\label{sin}
 \centering
 \includegraphics[scale=0.75]{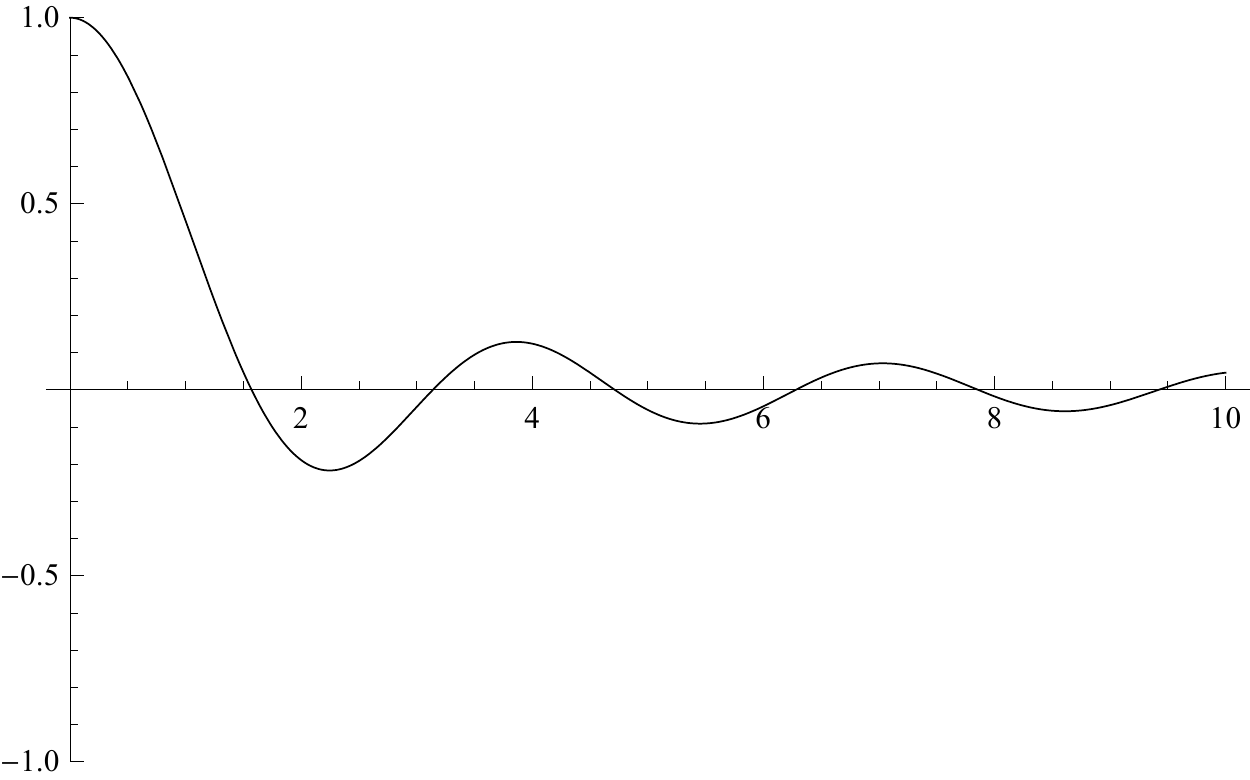}
 \caption{$\frac{\sin(2\alpha) }{2\alpha}$ on the interval $[0 \,, 10 ]$}
 \label{sin}
\end{figure}

\end{document}